\documentclass[11pt, twoside]{article}

\usepackage{amsfonts}
\usepackage{mathrsfs}
\usepackage{amssymb}
\usepackage{amsmath,amsthm}
\usepackage{color}

\allowdisplaybreaks

\newcount\quantno
\everydisplay{\quantno=0}\everycr{\quantno=0}
\newcommand\quant{\advance\quantno by1
                      \ifnum\quantno=1\qquad\else\quad\fi\forall }

\def\rr{{\mathbb R}}
\def\rn{{{\rr}^n}}

\def\cc{{\mathbb C}}
\def\nn{{\mathbb N}}
\def\zz{{\mathbb Z}}

\def\cm{{\mathcal M}}
\def\cn{{\mathcal N}}
\def\cp{\mathcal{D}}
\def\cq{{\mathcal Q}}
\def\crz{{\mathcal R}}

\def\fz{\infty}

\def\BMO{{\mathop\mathrm{\,BMO}}}

\def\lz{\lambda}

\def\kz{{\kappa}}
\def\tz{\theta}

\def\wz{\widetilde}

\def\ls{\lesssim}

\def\boz{\Omega}

\def\pz{{\prime}}

\def\gfz{\genfrac{}{}{0pt}{}}

\def\dint{\displaystyle\int}
\def\dfrac{\displaystyle\frac}

\def\r{\right}
\def\lf{\left}

\newtheorem{theorem}{Theorem}[section]
\newtheorem{lemma}[theorem]{Lemma}
\newtheorem{proposition}[theorem]{Proposition}
\newtheorem{corollary}[theorem]{Corollary}
\theoremstyle{definition}
\newtheorem{remark}{\it Remark}[section]
\newtheorem{definition}{Definition}[section]

\numberwithin{equation}{section}

\begin{document}

\arraycolsep=1pt

\title{{\vspace{-5cm}\small\hfill\bf Math. Z. (to appear)}\\
\vspace{4.5cm}\bf\Large Dyadic sets, maximal functions and applications on
$ax+b$\,--groups
\footnotetext{\hspace{-0.22cm}L. Liu $\cdot$ D. Yang (Corresponding author)\endgraf
School of Mathematical Sciences, Beijing Normal University,
Laboratory of Mathematics\endgraf
and Complex Systems, Ministry of
Education, Beijing 100875, People's Republic of China
\endgraf
e-mail: dcyang@bnu.edu.cn
\medskip
\endgraf
L. Liu
\endgraf
({\it Present address}) Department of Mathematics,
School of Information,
Renmin University \endgraf
of China, Beijing 100872, People's Republic of China\endgraf
e-mail: liguangbnu@gmail.com
\medskip\endgraf
M. Vallarino\endgraf
Dipartimento di Matematica e Applicazioni, Universit\`{a} di
Milano--Bicocca, via R. Cozzi 53,\endgraf
20125 Milano, Italy\endgraf
e-mail: maria.vallarino@unimib.it}}
\author{Liguang Liu, Maria Vallarino, Dachun Yang}
\date{}

\maketitle

\noindent{\bf Abstract}\quad Let $S$ be the Lie group
${\mathbb R}^n\ltimes {\mathbb R}$, where ${\mathbb R}$
acts on ${\mathbb R}^n$ by dilations,
endowed with the left-invariant
Riemannian symmetric space structure and the right Haar measure
$\rho$, which is a Lie group of exponential growth. Hebisch and
Steger in [Math. Z. 245(2003), 37--61] proved that any integrable
function on $(S,\rho)$ admits a Calder\'on--Zygmund decomposition
which involves a particular family of sets, called
Calder\'on--Zygmund sets. In this paper, we show the
existence of a dyadic grid in the group $S$, which has {nice} properties
similar to the classical Euclidean dyadic cubes. Using the
properties of the dyadic grid, we prove a
Fefferman--Stein type inequality, involving the dyadic Hardy--Littlewood
maximal function and the dyadic sharp function. As a consequence, we obtain a complex interpolation
theorem involving the Hardy space $H^1$ and the space
${\mathop\mathrm{\,BMO\,}}$
introduced in [Collect. Math. 60(2009), 277--295].

\medskip


\medskip


\medskip

\section{Introduction} \label{s: Introduction}

Let $S$ be the {\it Lie group} $\rn\ltimes \rr$ endowed with the
following {\it product}: for all $(x, t)$, $(x', t')\in S$,
$$(x, t)\cdot(x', t')\equiv (x+e^tx', t+t')\,.$$
The group $S$ is also called an
{\emph{$ax+b$--group}}. Clearly, $o= (0, 0)$ is the {\it identity} of $S$.
We endow $S$ with the {\it left-invariant Riemannian metric}
$$ds^2 \equiv e^{-2t}(dx_1^2+\cdots+dx_n^2)+dt^2\,,$$
and denote by $d$ the {\it corresponding metric}. This coincides with the
metric on the hyperbolic space $H^{n +1}(\rr)$. For all $(x, t)$ in
$S$, we have
\begin{equation}\label{1.1}
\cosh d\big((x, t), o\big)= \frac{  e^t+e^{-t}+e^{-t}|x|^2  }{2}.
\end{equation}
The group $S$ is nonunimodular. The {\it right and left Haar measures} are
given by
$$
d\rho(x, t)\equiv\,dx\,dt \quad {\rm{and}}
\quad d\lz(x, t)\equiv  e^{-nt}\,dx\,dt.
$$
Throughout the whole paper, we work on the {\it triple $(S, d, \rho)$},
namely, the group $S$ endowed with the left-invariant Riemannian
metric $d$ and the right Haar measure $\rho$. For all $(x, t)\in S$
and $r>0$, we denote by $B\big((x,t), r\big)$ the {\it ball centered at
$(x,t)$ of radius $r$}. In particular, it is well known that the
right invariant measure of the ball $B(o, r)$ has the following behavior
 \begin{eqnarray*}
\rho\big(B(o, r)\big)\sim\left\{\begin{array}{ll}
r^{n+1}\quad\quad&\text{if\quad $r<1$}\\
e^{nr}&\text{if\quad $r\ge 1$}.
\end{array}\right.
\end{eqnarray*}
Thus, $(S, d, \rho)$ is a space of exponential growth.

Throughout this paper, we denote by $L^p$ the {\it Lebesgue space $L^p(\rho)$} and by
$\|\cdot\|_{L^p}$ its {\it quasi-norm}, for all $p\in (0, \fz]$. We also
denote by $L^{1,\,\fz}$ the {\it Lorentz space $L^{1,\,\fz}(\rho)$}
and by $\|\cdot\|_{L^{1,\, \fz}}$ its {\it quasi-norm}.

Harmonic analysis on exponential growth groups recently attracts a
lot of attention. In particular, many efforts have been made to
study the theory of singular integrals on the space $(S,d,\rho)$.

In the remarkable paper \cite{hs}, Hebisch and Steger developed a new
Calder\'on--Zygmund theory which holds in some spaces of exponential
growth, in particular in the space $(S,d,\rho)$. The main idea of
\cite{hs} is to replace the family of balls which is used in the
classical Calder\'on--Zygmund theory by a suitable family of
rectangles which we call Calder\'on--Zygmund sets (see
Section \ref{s2} for their {definitions}). We let $\crz$
denote the {\it family of all Calder\'on--Zygmund sets}.

The Hardy--Littlewood maximal function associated with $\crz$ is of
weak type $(1,1)$ (see \cite{gs, va1}).
In \cite{hs}, it  was proven that every
integrable function on $(S,d,\rho)$ admits a Calder\'on--Zygmund
decomposition involving the family $\crz$. As a consequence, a
theory for singular integrals holds in this setting. In particular,
every integral operator bounded on $L^2$ whose kernel satisfies a
suitable integral H\"ormander's condition is of weak type $(1,1)$.
Interesting examples of singular integrals in this setting are
spectral multipliers and Riesz transforms associated with a
distinguished Laplacian $\Delta$ on $S$, which have been studied by
numerous authors in, for example, \cite{as, cghm, gqs, gas, hs, heb3, mt, sjo, sv}.

Vallarino \cite{va} introduced an atomic Hardy space $H^1$ on the
group {$(S, d,\rho)$,} defined by atoms supported in Calder\'on--Zygmund sets
instead of balls, and a corresponding $\BMO$ space, which enjoy some
properties of the classical Hardy and $\BMO$ spaces (see \cite{cw2, fest,
st93}). More precisely, it was proven that the dual of $H^1$ may be
identified with $\BMO$, that singular integrals whose kernel
satisfies a suitable integral H\"ormander's condition are bounded
from $H^1$ to $L^1$ and from $L^{\infty}$ to $\BMO$. Moreover,  for
every $\theta\in (0,1)$, the real interpolation space
$[H^1,L^2]_{\theta, q}$ is equal to $L^q$ if $\frac{1}{q} =1-\frac{\theta}{2}$,
and $[L^2,\BMO]_{\theta, p}$ is equal to $L^p$ if
$\frac{1}{p} = \frac{1-\theta}{2}$. The complex interpolation
spaces between $H^1$ and $L^2$ and between $L^2$ and $\BMO$
are not identified in \cite{va}.

In this paper, we introduce a dyadic grid of Calder\'on--Zygmund sets
on $S$, which we denote by $\cp$ and which can be considered as the
analogue of the family of classical dyadic cubes (see Theorem
\ref{t3.1} below). Recall that dyadic sets in the context of spaces
of homogeneous type were also introduced by Christ \cite{c}; his
construction used the doubling condition of the considered measure, so it
cannot be adapted to the current setting. {In the $ax+b$\,--groups, the main
tools we use to construct such a dyadic grid are some nice
splitting properties of the Calder\'on--Zygmund sets and an
effective method to construct a ``parent" of a given
Calder\'on--Zygmund set (see Lemma \ref{l3.1} below).
More precisely, given a Calder\'on--Zygmund set $R$, we
find a bigger Calder\'on--Zygmund set $M(R)$ which can be split into
at most $2^n$ sub-Calder\'on--Zygmund sets such that one of these
subsets is exactly $R$ and each of these subsets has measure
comparable to the measure of $R$.} To the best of our knowledge,
this is the first time that a family of dyadic sets appears in a
space of exponential growth. The dyadic grid $\cp$ turns out to be a
useful tool to study the analogue of maximal singular integrals (see
\cite{hyy}) on the space $(S,d,\rho)$, which will be investigated
in a forthcoming paper \cite{lvy}.

By means of the dyadic collection $\cp$, in Section 4 below, we
prove a relative distributional inequality
involving the dyadic Hardy--Littlewood maximal function and the
dyadic sharp maximal function on $S$, which implies a
Fefferman--Stein type inequality involving those
maximal functions; see Stein's book \cite[Chapter IV, Section
3.6]{st93} and Fefferman--Stein's paper \cite{fest} for
the analogous inequality in the Euclidean setting. The previous
inequality is the main ingredient to prove that the complex
interpolation space $(L^2,\BMO)_{[\theta]}$ is equal to
$L^{p_{\theta}}$ if $\frac{1}{p_{\theta}}= \frac{1-\theta}{2}$ and
$(H^1,L^2)_{[\theta]}$ is equal to $L^{q_{\theta}}$ if
$\frac{1}{q_{\theta}}= 1-\frac{\theta}{2}$. This implies complex
interpolation results for analytic families of operators
(see Theorems \ref{t5.2} and \ref{t5.3} below). {In particular,}
the complex interpolation result for analytic families of operators
involving $H^1$ could be interesting and useful to obtain endpoint
growth estimates of the solutions to the wave equation associated
with the distinguished Laplacian $\Delta$ on $ax+b$--groups, as was
pointed out {by M\"uller and Vallarino} \cite[Remark 6.3]{mv}.

We remark that the corresponding complex interpolation results for
the classical Hardy and $\BMO$ spaces were proven by Fefferman and
Stein \cite{fest}. Recently, an $H^1$--$\BMO$ theory was developed by
Ionescu \cite{i} for noncompact symmetric spaces of rank $1$ and,
more generally, by Carbonaro, Mauceri and Meda \cite{cmm} for metric
measure spaces which are nondoubling and satisfy suitable geometric
assumptions. In those papers, the authors proved a Fefferman--Stein type
inequality for the maximal functions associated with the family of balls
of small radius: the main ingredient in their proofs is an isoperimetric property
which is satisfied by the spaces studied in \cite{i, cmm}. As a consequence,
the authors in \cite{i, cmm} obtained some complex interpolation results
involving a Hardy space defined only by means of atoms supported in small balls
and a corresponding $\BMO$ space. Notice that the space $(S,d,\rho)$
which we study here does not satisfy the
isoperimetric property (\cite[(2.2)]{cmm}). Moreover, we
consider atoms supported both in {``small'' and ``big'' } sets.
Then we have to use different methods to obtain a suitable
Fefferman--Stein inequality and complex interpolation results
involving $H^1$ and $\BMO$.

Due to the existence of the dyadic collection $\cp$, it makes sense
to define a dyadic $\BMO_\cp$ space and its predual dyadic Hardy
space $H_\cp^1$ on $S$ (see Definitions
\ref{dyadic-hardy} and \ref{dyadic-bmo} below). Though in Theorem
\ref{dyadic-hardy and hardy} below, it is proven that $H_\cp^1$ is a
proper subspace of $H^1$, the complex interpolation result given by
$H_\cp^1$ and $L^2$ is the same as that given by $H^1$ and $L^2$;
see Remark \ref{dyadic-complex-interpolation} below.

Finally, we make some conventions on notations. Set
$\zz_+\equiv\{1,\,2,\,\cdots\}$ and $\nn=\zz_+\cup\{0\}$. In the
following, $C$ denotes a {\it positive finite constant} which may vary
from line to line and may depend on parameters according to the
context. Constants with subscripts do not change through the whole
paper. Given two quantities $f$ and $g$, by $f\ls g$, we mean that
there exists a positive constant $C$ such that  $f\le Cg$. If $f\ls
g\ls f$, we then write $f\sim g$. For any bounded linear operator
$T$ from a Banach space $A$ to a Banach space $B$, we denote by
$\|T\|_{A\to B}$ its {\it operator norm}.

\section{Preliminaries}\label{s2}

We first recall the definition of Calder\'on--Zygmund sets which appears
in \cite{hs} and implicitly in \cite{gs}. In the sequel, we  denote
by $\cq$ the {\it collection of all dyadic cubes in $\rn$}.

\begin{definition}\label{d2.1}
A {\it Calder\'on--Zygmund set} is a set $R\equiv Q\times[t-r, t+r)$, where
$Q\in\cq$ with side length $L$, $t\in\rr$, $r>0$ and
\begin{eqnarray*}
&&e^2\,e^tr\le L<e^8\,e^tr\quad\quad \mbox{if}\quad r<1,\\
&&e^t\,e^{2r}\le L<e^t\,e^{8r}\quad\quad\mbox{if}\quad r\ge1.
\end{eqnarray*}
We set $t_R\equiv t$, $r_R\equiv r$ and $x_R\equiv(c_Q, t)$, where
$c_Q$ is the {\it center} of $Q$. For a Calder\'on--Zygmund set $R$,
its {\it dilated set} is defined as $R^*\equiv \{x\in S:\, d(x, R)<r_R\}$.
Denote by $\crz$ the family of all Calder\'on--Zygmund sets on $S$.
For any $x\in S$, denote by $\crz(x)$
the family of the sets in $\crz$ which contain $x$.
\end{definition}

\begin{remark}\label{r2.1}\rm
For any set $R\equiv Q\times[t-r, t+r)\in\crz$, we have that
$$\rho(R)=\dint_Q\dint_{t-r}^{ t+r} \,ds\,dx=2r|Q|=2rL^n,$$
where $|Q|$  and $L$ denote the Lebesgue measure and the side length
of $Q$, respectively.
\end{remark}

The following lemma presents some properties of the
Calder\'on--Zygmund sets.

\begin{lemma}\label{l2.1}
Let all the notation be as in Definition \ref{d2.1}. Then there
exists a constant $\kz_0\in[1, \fz)$ such that for all $R\in\crz$,
the following hold:
\begin{itemize}
\vspace{-0.25cm}
\item[(i)] $B(x_R, r_R)\subset R\subset B(x_R, \kz_0r_R)$;
\vspace{-0.25cm}
\item[(ii)] $\rho(R^\ast)\le\kz_0\rho(R)$;
\vspace{-0.25cm}
\item[(iii)] every $R\in\crz$ can be decomposed into
mutually disjoint sets $\{R_i\}_{i=1}^k$, with $k=2$ or $k=2^n$,
$R_i\in \crz$, such that $R=\cup_{i=1}^k R_i$ and
$\rho(R_i)=\rho(R)/k$ for all $i\in\{1,\cdots, k\}$.
\end{itemize}
\end{lemma}

We refer the reader to \cite{hs, va} for the proof of Lemma \ref{2.1}.
We recall here an idea of the proof of the property (iii) from \cite{hs,va}.
Given a Calder\'on--Zygmund set $R\equiv Q\times [t-r,t+r)$,
when the side length $L$ of $Q$ is sufficiently large
with respect to $e^t$ and $r$, it suffices to decompose $Q$ into
$2^n$ smaller dyadic Euclidean cubes $\{Q_1,\cdots,Q_{2^n}\}$ and
define $R_i\equiv Q_i\times [t-r,t+r)$ for $i\in\{1,\cdots,2^n\}$. Otherwise, it suffices
to split up the interval $[t-r,t-r)$ into two disjoint sub-intervals
$\{I_1, I_2\}$, which have the same measure, and define
$R_i\equiv Q\times I_i$ for $i\in\{1, 2\}$. This construction gives rise
to either $2^n$ or $2$ smaller Calder\'on--Zygmund sets
satisfying the property (iii) above.

\smallskip

The Hardy--Littlewood maximal function associated to
the family $\crz$ is defined as follows.

\begin{definition}\label{d2.2}
For any locally integrable function $f$ on $S$, the {\it Hardy--Littlewood
maximal function $\cm f$} is defined by
\begin{equation}\label{2.1}
\cm
f(x)\equiv\sup_{R\in\crz(x)}\dfrac1{\rho(R)}\dint_R|f|\,d\rho\qquad
{ \forall\ } x\in S\,.
\end{equation}
\end{definition}
The maximal operator $\cm$ has the following boundedness properties
\cite{gs, hs, va1}.

\begin{proposition}\label{p2.1}
The Hardy--Littlewood maximal
operator $\cm$ is bounded from $L^1$ to
$L^{1,\,\fz}$, and also bounded on $L^p$ for all $p\in(1, \fz]$.
\end{proposition}

By Proposition \ref{p2.1}, Lemma \ref{l2.1} and a stopping-time argument,
Hebisch and Steger \cite{hs} showed that
any integrable function $f$ on $S$ at any level $\alpha>0$ has a
Calder\'on--Zygmund decomposition $f=g+\sum_ib_i$, where $|g|$ is a
function almost everywhere bounded by $\kz_0 \alpha$ and functions
$\{b_i\}_i$ have vanishing integral and are supported in sets of the
family $\crz$. This was proven to be a very useful tool
in establishing the boundedness of some multipliers and
singular integrals in \cite{hs}, and the theory of the
Hardy space $H^1$ on $S$ in \cite{va}.

Lemma {\ref{l2.1}(iii)} states that given a Calder\'on--Zygmund set, one
can split it up into a finite number of disjoint subsets which are
still in $\crz$. We shall now study how, starting from a given
Calder\'on--Zygmund set $R$, one can obtain a bigger set containing
it which is still in $\crz$ and whose measure is comparable to the
measure of $R$.
\begin{definition}\label{d3.1}
For any $R\in\crz$, $M(R)\in\crz$ is called a \emph{parent} of $R$,
if
\begin{itemize}
\vspace{-0.25cm}
\item[(i)] $M(R)$ can be decomposed into $2$ or $2^n$ mutually disjointed
sub-Calder\'on--Zygmund sets, and one of these sets is $R$;
\vspace{-0.25cm}
\item[(ii)] $3\rho(R)/2\le\rho(M(R))\le\max\{3, 2^n\}\rho(R)$.
\end{itemize}
\end{definition}

For any $R\in\crz$, a parent of $R$ always exists, but it may not be
unique. The following lemma gives three different kinds of
extensions for sets $R\equiv Q\times[t{-r},\, t+r)\in\crz$ when
$r\ge1$. Precisely, if $Q$ has small side length, then we find a
parent of $R$ by extending $R$ ``horizontally''; if $Q$ has large
side length, then we find a parent of $R$ by extending $R$
either ``vertically up" or ``vertically down".

\begin{lemma}\label{l3.1}
Suppose that $R\equiv Q\times[t{-r},\, t+r)\in\crz$, where $t\in\mathbb R$,
$r\equiv r_R\ge1$ and $Q\subset\rn$ is a dyadic cube with side
length $L$ satisfying $e^t\,e^{2r}\le L<e^t\,e^{8r}$. Then the following
hold:
\begin{itemize}
\vspace{-0.25cm}
\item[(i)] {If} $e^t\,e^{2r}\le L<e^t\,e^{8r}/2$, then $R_1\equiv
Q'\times[t-r,\, t+r)$ is a parent of $R$, where $Q'\subset\rn$
is the unique dyadic cube with side length $2L$ that contains $Q$.
Moreover, $\rho(R_1)=2^n\rho(R)$.
\vspace{-0.25cm}
\item[(ii)] If $e^t\,e^{8r}/2\le L<e^t\,e^{8r}$, then
$R_2\equiv Q\times[t{-r},\, t+{3r})$
is a parent of $R$. Moreover, the set $R'\equiv Q\times[t+r,
t+{3r})$ belongs to the family $\crz$, $R_2=R\cup R'$ and
$$\rho(R)=\rho(R')=\rho(R_2)/2.$$
The set $R_2$ is also a parent of $R'$.
\vspace{-0.25cm}
\item[(iii)] If $e^t\,e^{8r}/2\le L<e^t\,e^{8r}$, then
$R_3\equiv Q\times[t{-5r},\, t+{r})$
is a parent of $R$. Moreover, the set $R''\equiv Q\times[t{-5r},
t{-r})$ belongs to the family $\crz$, $R_3=R\cup R''$ and
$$\rho(R)=\rho(R'')/2=\rho(R_3)/3.$$
The set $R_3$ is also a parent of $R''$.
\end{itemize}
\end{lemma}
\begin{proof}
We first prove (i). Since $2e^te^{2r}\le 2L<e^te^{8r}$, we
have that $R_1\in\crz$. Obviously $Q'\times[t{-r},
t+r)$ can be decomposed into $2^n$ sub-Calder\'on--Zygmund sets and
one of these sets is $R$. By Remark \ref{r2.1},  we have
$\rho(R_1)=2^n\rho(R)$. Thus, (i) holds.

To show (ii), notice that $r_{R_2}=2r$ and $t_{R_2}=t+r$. Since
$r\ge 1$ and  $e^t\,e^{8r}/2\le L<e^t\,e^{8r}$, we have that $e^{t+r}\,e^{4r}\le
L<e^{t+r}\, e^{16r}$, which implies that $R_2\in\crz$. If we set
$$R'\equiv Q\times[t+r,\, t+3r),$$
then $t_{R'}=t+{2r}$ and
$r_{R'}=r$. Since $r\ge 1$ and $e^t\,e^{8r}/2\le
L<e^t\,e^{8r}$, we obtain $e^{t+2r}\,e^{r}\le L<e^{t+2r}\,e^{8r}$, and
hence $R'\in\crz$. By Remark \ref{r2.1},
$\rho(R)=\rho(R')=\rho(R_2)/2.$ Thus, (ii) holds.

Finally, we show (iii). Observe that $t_{R_3}=t-2r$ and
$r_{R_3}=3r$. Since $r\ge 1$ and $e^t\,e^{8r}/2\le
L<e^t\,e^{8r}$, we have that $e^{t-2r}\,e^{6r}\le L<e^{t-2r}\,e^{24r}$
and hence $R_3\in\crz$. Set $R''\equiv Q\times[t{-5r},\,
t{-r})$. Notice that $t_{R''}=t{-3r}$ and $r_{R''}=2r$. Again by
{$r\ge 1$}  and $e^t\,e^{8r}/2\le L<e^t\,e^{8r}$, we obtain that
$e^{t-3r}\,e^{4r}\le L<e^{t-3r}\,e^{16r}$ and hence $R''\in\crz$. It
is easy to see that $R_3=R\cup R''$, $\rho(R'')=2\rho(R)$ and
$\rho(R_3)=3\rho(R)$. Therefore, we obtain (iii), which completes
the proof.
\end{proof}

We conclude this section by recalling the definition of the Hardy
space $H^1$ and its dual space $\BMO$ (see \cite{va}).

\begin{definition}
An {\emph{$H^1$-atom}} is a function $a$ in $L^1$ such that
\begin{itemize}
\vspace{-0.2cm}
\item [(i)] $a$ is supported in a set $R\in\crz$;
\vspace{-0.2cm}
\item [(ii)]$\|a\|_{  L^{\infty} } \le [\rho(R)]^{-1};$
\vspace{-0.2cm}
\item [(iii)] $\int_S a \,d\rho =0$\,.
\end{itemize}
\end{definition}

\begin{definition}
The {\it Hardy space} $H^{1}$ is the space of all functions $g$ in $ L^1$
which can be written as $g=\sum_j \lambda_j\, a_j$, where
{$\{a_j\}_j$} are {$H^1$-atoms } and {$\{\lambda _j\}_j$} are
complex numbers such that $\sum _j |\lambda _j|<\infty$. {Denote} by
$\|g\|_{H^{1}}$ the {\it infimum of $\sum_j|\lambda_j|$ over such
decompositions}.
\end{definition}

In the sequel, for any locally integrable function $f$ and any set
$R\in \crz$, we denote by $f_R$ the {\it average of $f$ on $R$}, namely,
$\frac{1}{ \rho(R)} \int_R f d\rho$.

\begin{definition}
For any locally integrable function $f$, its {\it sharp maximal function}
is defined by
$$f^{\sharp}(x)\equiv\sup_{R\in\crz(x)}\frac{1}{ \rho(R)} \int_R
|f-f_R|\,d\rho \qquad { \forall\ } x\in S\,.$$
The space $\mathcal{B}\mathcal{M}\mathcal{O}$ is the {\it space of all
locally integrable functions $f$ such that $f^{\sharp}\in L^
{\infty}$}. The {\it space $\BMO$} is the quotient of
$\mathcal{B}\mathcal{M}\mathcal{O}$ module constant functions. It is
a Banach space endowed with the norm $\|f\|_{\BMO}
\equiv\|f^{\sharp}\|_{L^{\infty}}$.
\end{definition}

The space $\BMO$ is identified with the dual of $H^1$; see
\cite[Theorem 3.4]{va}. More precisely, for any $f$ in
$\BMO$, the functional $\ell$ defined by $\ell(g)\equiv\int
fg\,d\rho$ for any finite linear combination $g$ of atoms extends to
a bounded functional on $H^{1}$ whose norm is no more than $C\,\|f\|_{\BMO}.$
On the other hand, for any bounded linear functional $\ell$ on
$H^{1}$, there exists a function $f^{\ell}$ in $\BMO$
such that $\|f^{\ell}\|_{\BMO}\le C\,\|\ell\|_{(H^{1})^*}$ and
$\ell(g)=\int f^{\ell}g\,d\rho$ for any finite linear combination $g$ of
atoms.

\section{A dyadic grid on $(S, d, \rho)$}\label{s3}

The main purpose of this section is to introduce a dyadic grid of
Calder\'on--Zygmund sets on $(S, d,\rho)$, which can be considered
as an analogue of Euclidean dyadic cubes (see \cite[p.\,149]{st93}
or \cite[p.\,384]{g}). The key tools to construct  such a grid are
{Lemmas \ref{l2.1} and \ref{l3.1}.}

\begin{theorem}\label{t3.1}
There exists a collection $\{\cp_j\}_{j\in\zz}$ of partitions of $S$
such that each $\cp_j$ consists of pairwise disjoint
Calder\'on--Zygmund sets, and
\begin{itemize}
\vspace{-0.25cm}
\item[(i)] for any $j\in\zz$,  $S=\cup_{R\in\cp_j} R$;
\vspace{-0.25cm}
\item[(ii)] if $\ell\le k$, $R\in\cp_\ell$ and $R^\pz\in\cp_k$, then
either $R\subset R^\pz$ or $R\cap R^\pz=\emptyset$;
\vspace{-0.25cm}
\item[(iii)] for any $j\in\zz$ and $R\in\cp_j$, there
exists a unique $R^\pz\in\cp_{j+1}$ such that $R\subset R^\pz$ and
$\rho(R^\pz)\le 2^n\,\rho(R)$;
\vspace{-0.25cm}
\item[(iv)] for any $j\in\zz$, every $R\in\cp_j$ can be decomposed into
mutually disjoint sets $\{R_i\}_{i=1}^k\subset\cp_{j-1}$, with $k=2$
or $k=2^n$, such that $R=\cup_{i=1}^k R_i$ and
$\frac{\rho(R)}{2^n}\le\rho(R_i)\le {\frac{2\rho(R)}{3}}$ for all
$i\in\{1,\cdots,k\}$;
\vspace{-0.25cm}
\item[(v)] for any $x\in S$ and for any $j\in\zz$,
let $R_j^x$ be the unique set in $\cp_j$ which contains $x$, then
$\lim_{j\to -\infty}\rho(R_j^x)=0$ and $\lim_{j\to
\infty}\rho(R_j^x)=\infty$.
\end{itemize}
\end{theorem}

\begin{proof}
We write $S\equiv\boz_1\cup\boz_2$, where $\boz_1\equiv\rn\times[0,
\fz)$ and  $\boz_2\equiv\rn\times(-\infty,0)$,
and construct a sequence $\{\cp_j^1\}_{j\in\nn}$
of partitions of $\boz_1$ as well as a sequence $\{\cp_j^2\}_{j\in\nn}$
of partitions of $\boz_2$, respectively.

Let us first construct the desired partitions
$\{\cp_j^1\}_{j\in\nn}$ of $\boz_1$ by the following four steps.

{{\bf Step 1.} } Choose a Calder\'on--Zygmund set $R_0\equiv
Q_0\times[t_0-r_0, t_0+{r_0})$, where $t_0=r_0\ge1$ and
$Q_0=[0, \ell_0)^n\in\cq$, with $e^{t_0}\,e^{2
r_0}\le\ell_0<e^{t_0}\,e^{8 r_0}$. To find a parent of $R_0$, we consider
the following two cases, separately.

{\it {Case 1:}} $e^{t_0}\,e^{8 r_0}/2\le\ell_0<e^{t_0}\,e^{8 r_0}$. In this
case,  by Lemma \ref{l3.1}(ii),
$$R_1\equiv Q_0\times[t_0{-r_0}, t_0+{3r_0})$$
is a parent of $R_0$ and $\rho(R_1)=2\rho(R_0)$.

{\it {Case 2:}} $e^{t_0}\,e^{2 r_0}\le\ell_0< e^{t_0}\,e^{8 r_0}/2$.
By Lemma \ref{l3.1}(i), $R_1\equiv Q_1\times[t_0{-r_0},
t_0+{r_0})$ is a parent of $R_0$, where $Q_1\subset\rn$ is the
unique dyadic cube with side length $2\ell_0$ that contains $Q_0$,
namely,  $Q_1=[0, 2\ell_0)^n$.

We then proceed as above to obtain a parent of $R_1$, which is
denoted by $R_2$. By repeating this process, we obtain a sequence of
Calder\'on-Zygmund sets, $\{R_j\}_{j\in\nn}$, such that each
$R_{j+1}$ is a parent of $R_j$.

Without loss of generality, for any $j\in\nn$, we may set $R_j\equiv
Q_j\times [t_j-r_j, t_j+r_j)$, where $r_{j+1}\ge r_j\ge1$,
$t_j=r_j$ and $Q_j=[0, \ell_j)^n\in\cq$, with
$e^{t_j}\,e^{2 r_j}\le\ell_j<e^{t_j}\,e^{8 r_j}$. Observe that $R_{j+1}$ is
obtained by extending $R_j$ either ``vertically up" (see {\it {Case
1}}) or ``horizontally" (see {\it {Case 2}}). Notice that the
definition of Calder\'on--Zygmund sets implies that we cannot always
{extend} $R_j$ ``horizontally'' to obtain its parent $R_{j+1}$; in
other words, for some $j$, to obtain $R_{j+1}$, we have to {extend}
$R_j$ ``vertically {up}''. Thus, $\lim_{j\to\fz}(t_j+r_j)=\fz$.
This, combined with the fact that $t_j=r_j$, implies that
\begin{equation}\label{3.1}
\boz_1 = \bigcup_{j\in\nn}  \big(\rn\times[t_j-r_j, t_j+{r_j})\big).
\end{equation}

{\bf Step 2. } For any $j\in\nn$ and $R_j$ as constructed in {\bf
Step 1}, we set
\begin{equation}\label{3.2}
\cn_j\equiv\{Q\times[t_j{-r_j}, t_j+{r_j}):\, Q\in\cq,\,
\ell(Q)=\ell(Q_j)\}.
\end{equation}
Then $\cn_j\subset\crz$ and we put all sets of $\cn_j$ into
$\cp_j^1$. If $R_{j+1}$ is obtained by extending $R_j$ ``vertically
{up}" as in {\it Case 1} of {\bf Step 1}, then we set
\begin{equation}\label{3.3}
\wz{\cn_j}\equiv\{Q\times[t_j+{r_j}, t_j+{3r_j}):\, Q\in\cq,\,
\ell(Q)=\ell(Q_j)\}.
\end{equation}
By
Lemma \ref{l3.1}(ii), $\wz{\cn_j}\subset\crz$.
If $R_{j+1}$ is obtained by extending $R_j$ ``horizontally" as in
{\it Case 2} of {\bf Step 1}, then we set $\wz{\cn_j}=\emptyset$.
We also put all sets of $\wz{\cn_j}$ into $\cp_j^1$.

We claim that for any fixed $j\in\nn$,
\begin{equation}\label{Omega1unione}
\Omega_1=\bigcup_{R\in\cn_j\cup (\cup_{\ell=j}^{\infty}\wz{\cn_{\ell}})} R\,.
\end{equation}
Indeed,
\begin{equation}\label{primaparte}
\rr^n\times [0,t_j+{3r_j})=\bigcup_{R\in \cn_j\cup \wz{\cn_j}}R.
\end{equation}
Rewrite the sequence $\{\wz\cn_k:\, k>j,\, \wz{\cn}_{k}\neq\emptyset\}$ as
$\{\wz{\cn}_{\ell_k}\}_{k=1}^\infty$, where $$j+1\le
\ell_1<\ell_2<\dots<\ell_k<\dots \,.$$
We have that
\begin{equation} \label{oss1}
t_j+{3r_j}=t_{\ell_1}+{r_ {\ell_1} } \quad {\rm{and}} \quad
t_{\ell_{k-1}}+{3r_ {\ell_{k-1}} }=t_{\ell_k}+{r_ {\ell_k} } \quad
 \forall\ k\ge 1.
\end{equation}
Since
$$
\rr^n\times [t_{\ell_k}+{r_ {\ell_k} }, t_{\ell_k}+{3r_ {\ell_k}
})= {\bigcup_ {R\in\wz{\cn} _{\ell_k} } R }
$$
and $\lim_{k\to \infty}  ( t_{\ell_k}+{3r_ {\ell_k}})
={\infty}$, by (\ref{oss1}), we obtain that
\begin{equation}\label{secondaparte}
{ \rr^n\times [ t_j+{3r_j},\infty ) =\bigcup_{k\ge
1}\bigcup_{R\in\wz{\cn}_{\ell_k}  }  R = \bigcup_{
\ell\ge j+1} \bigcup_{R\in \wz{\cn}_{\ell} }   R.}
\end{equation}
The claim \eqref{Omega1unione} follows by \eqref{primaparte} and
\eqref{secondaparte}.

{\bf Step 3. } Now fix $j\in\nn$ and take $\ell\ge j+1$ such that
$\wz{\cn_{\ell}}\neq\emptyset$.
For any
$R\in\wz{\cn_{\ell}}$, by Lemma \ref{l2.1}(iii), there exist mutually
disjoint sets $\{R^{i}\}_{i=1}^k\subset\crz$ with $k=2$ or $k=2^n$
such that $R=\cup_{i=1}^k R_i$, and
$\rho(R)/2^n\le\rho(R_{i})\le\rho(R)/2$ for all $i\in\{1,\cdots,k\}$.
Denote by $\wz{\cn_{\ell}}^1$ the collection of all such small
Calder\'on-Zygmund sets $R_i$ obtained by running $R$ over all
elements in $\wz{\cn_{\ell}}$. Observe that sets in $\wz{\cn_{\ell}}^1$ are
mutually disjoint. Next, we apply Lemma \ref{l2.1}(iii) to every
$R\in\wz{\cn_{\ell}}^1$ and argue as above; we then obtain a  collection
of smaller Calder\'on--Zygmund sets, which is denoted by
$\wz{\cn_{\ell}}^2$. By repeating the above procedure $i$ times, we
obtain a collection of Calder\'on--Zygmund sets which we denote by $\wz{\cn_\ell}^{i}$.
In particular, we put the collection $\wz{\cn_\ell}^{\ell-j}$ obtained after $\ell-j$
steps into $\cp_j^1$.

Thus, for any $j\in\nn$, we define
\begin{equation}\label{3.4}
\cp_j^1=\cn_j\bigcup\wz{\cn_j} \bigcup   \lf(\bigcup_{\ell\ge
j+1}\wz{\cn_\ell}^{\ell-j}\r)   .
\end{equation}
By construction, the sets in $\cp_j^1$ are mutually disjoint. Moreover,
since for all $j\ge 0$ and $\ell\ge j+1$,
$$
\bigcup_{R\in\wz{\cn}_{\ell}^{\ell-j}}R=\bigcup_{R\in\wz{\cn}_{\ell}}R\,,
$$
from the formula \eqref{Omega1unione},
we deduce that $\boz_1=\cup_{R\in\cp_j^1}R$. This
shows that $\cp_j^1$ satisfies the property (i).

{\bf Step 4.} For any $0\le \ell\le k$, $R\in\cp_\ell^1$ and
$R^\pz\in\cp_k^1$, by \eqref{3.4}, \eqref{3.2}, \eqref{3.3} and the
construction above, it is easy to verify that either
$R\subset R^\pz$ or $R\cap R^\pz=\emptyset$, namely, the property
(ii) is satisfied.

Let $R$ be in $\cp_j^1$ for some $j\in\nn$. If $R$ is in
$\cn_j\cup\wz{\cn}_j$ and if $R_{j+1}$ is obtained by extending
$R_j$ ``horizontally", then there exists one parent of $R$ in
$\cp_{j+1}$ whose measure is {$2^n\rho(R)$}. If $R$ is in
$\cn_j\cup\wz{\cn}_j$ and if $R_{j+1}$ is obtained by extending
$R_j$ ``vertically {up}", then there exists one parent of $R$ in
$\cp_{j+1}$ whose measure is {$2\rho(R)$.} If $R$ is in
$\wz{\cn}_{\ell}^{\ell-j}$ for some $\ell\ge j+1$, then it has a
parent in $\wz{\cn}_{\ell}^{\ell-j-1} \subset \cp_{j+1}$ whose
measure is either $2\rho(R)$ or $2^n\rho(R)$. Thus, the property (iii) is
satisfied.

So far, we have proven that there exists a sequence $\{\cp_j^1\}_{j\in\nn}$
of partitions of $\boz_1$ whose elements satisfy the
properties (i)--(iii).

To obtain the desired partitions $\{\cp_j^2\}_{j\in\nn}$ on
$\boz_2$, we {apply (i) and (iii) of Lemma \ref{l3.1} and} proceed
as for $\boz_1$: the details are left to the reader.

We define $\cp_j\equiv\cp_j^1\cup\cp_j^2$ for all $j\ge 0$.
We now construct the partitions $\cp_j$ for $j<0$. By applying
Lemma \ref{l2.1}(iii) to each $R\in\cp_0$, we find mutually disjoint sets
$\{R_{i}\}_{i=1}^k$, with $k=2$ or $k=2^n$, such that $R_i\in\crz$,
$R=\cup_{i=1}^k R_i$, and $\rho(R)/2^n\le\rho(R_{i})\le\rho(R)/2$
for all $i\in\{1,\cdots,k\}$. Then we define $\cp_{-1}$ to be the
collection of all such small Calder\'on-Zygmund sets $R_i$ obtained
by running $R$ over all elements in $\cp_0$. Clearly $\cp_{-1}$ is
still a partition of $S$. Again, applying {Lemma \ref{l2.1}(iii)} to
each element of $\cp_{-1}$ and using a similar splitting argument
to this, we obtain a collection of smaller Calder\'on--Zygmund sets, which is
defined to be $\cp_{-2}$. By repeating this process, we obtain a
collection $\{\cp_j\}_{j<0}$, where each $\cp_j$ is a partition of
$S$. By {the} construction of $\{\cp_j\}_{j<0}$ and by Lemma \ref{l2.1}(iii),
it is easy to check that the sets in
$\{\cp_j\}_{j<0}$ satisfy {the} properties {(i)--(iii).}

It remains to prove the properties (iv) and (v). For a set $R\in \cp_j$,
with $j\le 0$, the property (iv) is easily deduced from Lemma \ref{l2.1}(iii).
Take now a set $R$ in $\cp_j^1$ for some $j>0$. If $R$ is in $\cn_j$,
then it has either $2^n$ disjoint subsets in $\cn_{j-1}$ or $2$
disjoint subsets in $\cn_{j-1}\cup  \wz{\cn}_{j-1}$. If $R$ is in
$\wz{\cn}_j$, then it has either $2^n$ or $2$ disjoint subsets in
$\wz{\cn}_{j}^1\subset \cp_{j-1}^1$. Finally, if $R$ is in
$\wz{\cn}_{\ell}^{\ell-j}$ for some $\ell\ge j+1$, then it has
either $2^n$ or $2$ subsets in $\wz{\cn}_{\ell}^{\ell-j+1}\subset
\cp_{j-1}^1$. In all the previous cases, $R$ satisfies the property (iv).
The case when $R$ is in $\cp_j^2$ for some $j>0$ is similar and omitted.

As far as the property (v) is concerned, given a point $x$ in $S$, for any
$j\in\zz$, let $R_j^x$ be the set in $\cp_j$ which contains $x$. By
the construction and the property (iv), for any $j\in\zz$,
there exists a set
$R_{j+1}^x\in \cp_{j+1}$ which is a parent of $R_j^x$, so that
$$\rho( R_{j+1}^x)\ge \frac{3}{2}\rho(R_j^x)\ge
\lf(\frac{3}{2}\r)^j\rho(R_0^x);$$
this shows that
$\lim_{j\to \infty}  \rho(R_j^x)=\infty$.
For any $j<0$,
we have that
$$\rho(R_j^x)\le {\frac23}\rho(R_{j+1}^x) \le
{\lf(\frac23\r)^j}\rho(R_0^x);$$
this shows that $\lim_{j\to -\infty}
\rho(R_j^x)=0$ and concludes the proof of the theorem.
\end{proof}

\begin{remark}\label{r3.1}
\begin{itemize}
\item[(i)] It should be pointed out that  a sequence $\{\cp_j\}_{j\in\zz}$
satisfying Properties (i)--(v) of Theorem \ref{t3.1} is not unique.
\vspace{-0.25cm}
\item[(ii)] For any given $j\in\zz$, the measures of any two elements
in $\cp_j$ may not be comparable. This is an essential difference
between the collection of Euclidean dyadic cubes and of dyadic sets
in spaces of homogeneous type \cite{c} and the dyadic sets which we
introduced above.
\end{itemize}
\end{remark}

We now choose one collection $\cp\equiv \{\cp_j \}_j$ of dyadic sets
in $S$ constructed as in Theorem \ref{t3.1}. In the sequel, $\cp$
always denotes this collection.

\section{Dyadic maximal functions}

By using the collection $\cp$ introduced above, we define the
corresponding Hardy--Littlewood dyadic maximal function and dyadic
sharp maximal function as follows.

\begin{definition}\label{dyadicmaxfct}
For any locally integrable function $f$ on $(S, d, \rho)$,
the {\it Hardy--Littlewood dyadic maximal function} $\cm_{\cp} f$
is defined by
\begin{equation}\label{3.6}
\cm_{\cp} f(x)\equiv\sup_{R\in\crz(x),\,
R\in\cp}\dfrac1{\rho(R)}\dint_R|f|\,d\rho\qquad {\forall\ } x\in
S\,,
\end{equation}
and the {\it dyadic sharp maximal function} $f^{\sharp}_{\cp}$ by
\begin{equation}\label{fsharpd}
f^{\sharp}_{\cp}(x)\equiv\sup_{R\in\crz(x),\,
R\in\cp}\dfrac1{\rho(R)}\dint_R|f -f_R|\,d\rho\qquad {\forall\ }
x\in S\,.
\end{equation}
Recall that $f_R\equiv\frac{1}{ \rho(R)} \int_R f\, d\rho$.
\end{definition}

It is easy to see that for all locally integrable functions $f$ and
almost every $x\in S$, $ f(x)\le\cm_{\cp} f(x)\le\cm f(x)$ {and
$f_\cp^\#(x)\le f^\#(x)$.} This combined with Proposition \ref{p2.1}
implies the following conclusion.

\begin{corollary}\label{c3.1}
The operator $\cm_{\cp}$ is bounded from $L^1$ to $L^{1,\,\fz}$, and
also bounded on $L^p$ for all $p\in(1, \fz]$.
\end{corollary}

\begin{remark}
\begin{itemize}
\item[(i)] It is obvious that $\cm_{\cp} f(x)\le  \cm f(x)$
for any locally integrable function $f$ at any point $x\in S$.
However, there exist functions $f$ such that $\cm f$ and $\cm_{\cp}
f$ are not pointwise equivalent. To see this, we take a set
$R\equiv \,Q\times  [0, {2r})$ in $\cp$ such that
$Q={[0, 2^{\ell_0})}^n$ for some $\ell_0\in\zz$. Then, for all
points $(y_1,\dots,y_n, s)\in S$ such that $y_j<0$ for all $j\in
\{1,\dots,n\}$ and $s<0$, we have $\cm_{\cp}(\chi_R)(y, s)=0$
and $\cm(\chi_R)(y, s)>0.$ So there does not exist a positive
constant $C$ such that $\cm(\chi_R)\le C \cm_{\cp}(\chi_R)$.
\vspace{-0.25cm}
\item[(ii)] It is obvious that $f^{\sharp}_{\cp}(x)\le
f^{\sharp}(x)$ for any locally integrable function $f$
at any point $x\in S$. The same counterexample as in (i) shows that
the sharp maximal function and the dyadic sharp maximal function may
be not pointwise equivalent. Indeed, if we take the set
$R\equiv \,Q\times  [0, 2r)$ as above, then for all points
$(y_1,\dots,y_n, s)\in S$ such that $y_j<0$ for all $j\in
\{1,\dots,n\}$ and $s<0$, we have
$(\chi_R)^{\sharp}_{\cp} (y, s)=0$ and $(\chi_R)^{\sharp} (y, s)>0.$
So there does not exist a positive constant $C$ such that
$(\chi_R)^{\sharp} \le C {(\chi_R)^{\sharp}_\cp}$.
\end{itemize}
\end{remark}

We now state a covering lemma for the level sets of $\cm_{\cp}$,
which is proven in a standard way as follows;
see also \cite[Lemma 1, p.150]{st93}.

\begin{lemma}\label{covering}
Let $f$ be a locally integrable function and $\alpha$ a positive
constant such that
$\Omega_{\alpha}\,\equiv \{x\in S:\, \cm_{\cp}f(x)>\alpha\}$
has finite measure. Then $\Omega_t$ is a disjoint
union of dyadic sets, $\{R_j\}_j$, with
$\alpha<\frac{1}{\rho(R_j)}\int_{R_j} |f|d\rho\le 2^n\alpha$
for all $j$.
\end{lemma}

\begin{proof}
Since the measure of $\Omega_{\alpha}$ is finite,
for each $x\in\Omega_{\alpha}$
there exists a maximal dyadic set $R_x\in\cp$ which contains $x$ such
that $ \alpha<\frac{1}{\rho(R_x)}\int_{R_x} |f|d\rho$. Any two of these
maximal dyadic sets are disjoint. Indeed, {by Theorem \ref{t3.1},}
given two points $x,y\in \Omega_{\alpha}$, either $R_x\cap R_y=\emptyset$
or one is contained in the other; by maximality, this implies
that $R_x=R_y$. We denote by $\{R_j\}_j$ this collection of
dyadic maximal sets. Then it is clear that $\Omega_{\alpha}=\cup_jR_j$.
Moreover, for any $j$, since $R_j$ is maximal, there exists a dyadic
set $\tilde{R}_j\in\cp$ which {is a parent of $R_j$ and}
$\frac{1}{\rho(\tilde{R}_j)}\int_{\tilde{R}_j} |f|\, d\rho\le {\alpha}$.
{Thus,}
$$\frac{1}{\rho(R_j)}\int_{R_j} |f|\,d\rho \le
2^n\frac{1}{\rho(\tilde{R}_j)}\int_{\tilde{R}_j} |f|\,d\rho\le
2^n{\alpha}.$$
This finishes the proof.
\end{proof}

As a consequence of the previous covering lemma, following closely
the proof of the inequality \cite[(22), p.153]{st93}, we obtain the
following relative distributional inequality. We omit the details.

\begin{proposition}
There exists a positive constant $K$ such that for any locally
integrable function $f$, and for any positive $c$ and $b$ with
$b<1$,
\begin{equation}\label{distr}
\rho\big( \{  x\in S:\,\cm_{\cp}f(x)>\alpha,\,
f^{\sharp}_{\cp}(x)\le c\alpha\}\big) \le K \frac{c}{   1-b} \,
\rho\big( \{  x\in S:\,\cm_{\cp}f(x)>b\alpha\}\big)
\end{equation}
for all $\alpha>0$. The constant $K$ only depends on $n$ and on the
norm $\|\cm_{\cp}\|_{L^1\to L^{1,\infty} }$.
\end{proposition}

By the relative distributional inequality (\ref{distr}) and arguing as
in {\cite[Corollary 1, p.\,154]{st93}}, we obtain the following
Fefferman--Stein type inequality. We also omit the details.

\begin{corollary}\label{FeffStein}
Let {$p\in(0,\infty)$.} There exists a positive constant $A_p$ such
that for any locally integrable function $f$ such that
$f^{\sharp}_{\cp}$ belongs to $L^p$ and {$\cm_{\cp}f \in L^{p_0}$}
with $p_0\le p$, then $f$ is in $L^p$ and
$$\| \cm_{\cp}f \|_{L^p}\le A_p\,\| f^{\sharp}_{\cp} \|_{L^p}.$$
\end{corollary}

\begin{remark}
Recall that for any locally integrable function $f$, $|f|\le
\cm_{\cp}f$ and $f^{\sharp}_{\cp}\le f^{\sharp}$. Thus, from
Corollary \ref{FeffStein}, we deduce that if $p\in {(0,\infty)}$,
$f^{\sharp}$ belongs to $L^p$ and $f$ belongs to some $L^{p_0}$
with $p_0\in (0, p]$, then $f$ is in $L^p$ and
\begin{equation}\label{FeffStein2}
\| f \|_{L^p}\le A_p\,\| f^{\sharp} \|_{L^p}\,,
\end{equation}
where $A_p$ is the constant which appears in Corollary
\ref{FeffStein}. This generalizes the classical Fefferman--Stein
inequality {\cite[Theorem 2, p.148]{st93}} to the current setting.
\end{remark}

We shall now introduce a dyadic Hardy space and a dyadic $\BMO$
space.

\begin{definition}\label{dyadic-hardy}
The {\it dyadic Hardy space} $H^{1}_{\cp}$ is defined
to be the space of all functions $g$ in $ L^1$ which
can be written as $g=\sum_j \lambda_j\, a_j$,
where $\{a_j\}_j$ are $H^1$-atoms supported in dyadic sets and
{$\{\lambda _j\}_j$} are complex numbers such that $\sum _j |\lambda
_j|<\infty$. Denote by $\|g\|_{H^{1}_{\cp}}$ the {\it infimum of
$\sum_j|\lambda_j|$ over all such decompositions}.
\end{definition}

\begin{definition}\label{dyadic-bmo}
The space $\mathcal{B}\mathcal{M}\mathcal{O}_{\cp}$ is the space of
all locally integrable functions $f$ such that $f^{\sharp}_{\cp}\in
L^ {\infty}$. The space $\BMO_{\cp}$ is the quotient of
$\mathcal{B}\mathcal{M}\mathcal{O}_{\cp}$ module constant functions.
It is a Banach space endowed with the norm $\|f\|_{\BMO_{\cp}}
\equiv\|f^{\sharp}_{\cp}\|_{L^\infty}$.
\end{definition}

It is easy to follow the proof in \cite[Theorem 3.4]{va1} to show
that the dual of $H^1_{\cp}$ is identified with $\BMO_{\cp}$.
We omit the details.

Obviously, $H^1_{\cp}\subset H^1$ and $\|g\|_{H^1}\le
\|g\|_{H^1_{\cp}}$ for all $g$ in $H^1_\cp$. It is
natural to ask whether the norms $\|\cdot\|_{H^1}$ and
$\|\cdot\|_{H^1_{\cp}}$ are equivalent. The analog problem in the
classical setting was studied by {Abu-Shammala and Torchinsky}
\cite{at}. By following the ideas in \cite{at}, we
obtain the following result.

\begin{theorem}\label{dyadic-hardy and hardy}
The norms $\|\cdot\|_{H^1}$ and $\|\cdot\|_{H^1_{\cp}}$ are not
equivalent.
\end{theorem}

\begin{proof}
We give the details of the proof in the case when $n=1$.
By the construction of
$\cp$ in Theorem \ref{t3.1}, there exists $[0,2^{\ell_0+1})\times
[0,2{r_0})\in\cp_{k_0+1}$ for some $k_0\in\zz$, $\ell_0\in\zz$
and $r_0>0$ such that $R_0\equiv[0,2^{\ell_0})\times
[0,2{r_0})\in \cp_{k_0}$ and $E_0\equiv[2^{\ell_0}, 2\cdot
2^{\ell_0})\times [0,2{r_0})\in \cp_{k_0}$. Generally,  for any
$j<0$, there exist
$R_j=[2^{\ell_0}-2^{\ell_j},2^{\ell_0})\times I_j\in\cp_{k_j}$ and
$E_j=[2^{\ell_0},2^{\ell_0}+2^{\ell_j})\times I_j\in\cp_{k_j}$ such
that $R_j\cup E_j\in \cp_{k_j+1}$, where both $\{k_j\}_{j<0}$ and
$\{\ell_j\}_{j<0}$ are strictly decreasing sequences which tend to
$-\infty$ as $j\to-\infty$, and each $I_j$ is an interval contained in
$[0,\infty)$. Notice that for all $j\in\nn$,
$\rho(R_j)=\rho(E_j)=2r_j2^{\ell_j}$ for some $r_j>0$.
Set $a_j \equiv \frac 1{2\rho(R_j)} (\chi_{R_j}-\chi_{E_j})$.
Obviously, each $a_j$
is an $H^1$-atom and $\|a_j\|_{H^1}\le 1$.

Take the function $\phi(x,t)\equiv\chi_{(2^{\ell_0},\infty)}(x)
\log(x-2^{\ell_0})\equiv h(x)$ for all $(x,t)\in S$.
An easy calculation gives that
$$\|\phi\|_{\BMO_\cp}\le \sup_{\gfz{I\subset\rr }
{ I\, \mathrm{is\, a\, dyadic\, interval}} } \frac1{|I|}\int_I
\lf|h(x)-\frac1 {|I|}\int_I h(y)\,dy\r|\,dx<\infty.$$ We then have
\begin{equation*}
 \begin{aligned}
\|a_j\|_{H^1_{\cp}} & = \sup_{\psi\in \BMO_{\cp}}
\frac 1{\|\psi\|_{\BMO_{\cp}}}\lf|\int_S a_j\,\psi\, d\rho\r|\\
&\ge \frac 1{\|\phi\|_{\BMO_{\cp}}}\lf|\int_S a_j\,\phi\, d\rho\r|\\
&=\frac 1{2\|\phi\|_{\BMO_{\cp}}}\lf|2^{-\ell_j}\int_{2^{\ell_0}}^{
2^{\ell_0}+2^{\ell_j}}\,\log(x-2^{\ell_0})\,dx\r|
=\frac{|1-\log 2^{\ell_j}|}{2\|\phi\|_{\BMO_{\cp}}} \sim
|\ell_j|\,.
\end{aligned}
\end{equation*}
So there exists no positive constant such that
$\|a_j\|_{H^1_{\cp}}\le C\|a_j\|_{H^1}$ for all $j<0$.
\end{proof}

Notice that all the arguments of \cite[Section 5]{va}
can be adapted to the dyadic spaces $H^1_{\cp}$ and $\BMO_{\cp}$
such that all results therein also hold for $H^1_{\cp}$ and $\BMO_{\cp}$. In
particular, one can prove that, though $H_\cp^1$ is a proper
subspace of $H^1$, the real interpolation space
$[H^1_{\cp},L^2]_{\theta,q}$ is equal to $L^q$, if $\theta\in (0,1)$
and $\frac{1}{q}=1-\frac{\theta}{2}$.

\section{Complex interpolation}

We now formulate an interpolation theorem involving $H^1$ and $\BMO$.
In the following, when $A$ and $B$ are Banach spaces and $\theta$ is
in $(0,1)$, we denote by $(A,B)_{[\theta]}$ the complex
interpolation space between $A$ and $B$ with parameter $\theta$,
obtained via Calder\'on's complex interpolation method
(see \cite{Ca, bl}).

\begin{theorem}\label{int1}
Suppose that $\theta$ is in $(0,1)$. Then the following hold:
\begin{itemize}
\vspace{-0.25cm}
 \item[(i)] if $\frac{1}{p_{\theta}}= \frac{1-\theta}{2}$,
 then $(L^2,\BMO)_{[\theta]}=L^{p_{\theta}};$
 \vspace{-0.25cm}
\item[(ii)] if $\frac{1}{q_{\theta}}=1- \frac{\theta}{2}$,
then $(H^1,L^2)_{[\theta]}=L^{q_{\theta}}$.
\end{itemize}
\end{theorem}

\begin{proof}
 The proof of (i) is an easy adaptation of the proof of
{\cite[p.156,\,Corollary 2]{fest}} and of
\cite[Theorem 7.4]{cmm}. We omit the
details.

The proof of (ii) follows from a duality argument
(see \cite[Theorem 7.4]{cmm}). Denote by $X_{\theta}$
the interpolation space $\bigl(H^1, L^2 \bigr)_{[\theta]}$.
Now by the duality theorem \cite[Corollary 4.5.2]{bl}, if
$\frac{1}{q_{\theta}}=1-\frac{\theta}{2}$, then the dual of
$X_{\theta}$ is $\bigl(L^2,\BMO\bigr)_{[\theta]}$, which is equal to
$L^{q_\theta'}$ by (i), where $\frac 1{q_\tz}+\frac 1{q_\tz'}=1$.
Furthermore, $X_\theta$ is continuously
included in $L^{q_\theta}$, because $H^1$ is continuously included
in $L^1$ and $\bigl(L^1,L^2\bigr)_{[\theta]} =
L^{q_\theta}$. Since $L^2$ is reflexive, the interpolation space
$X_\theta$ is reflexive (see \cite[Section~4.9]{bl}), so that
$X_{\theta}$ is isomorphic to $X_\theta^{**} =
\bigl(L^{q_\theta'}\bigr)^*=L^{q_\theta}$. This concludes the
proof.
\end{proof}

A consequence of the previous theorem is the following.
\begin{theorem}\label{t5.2}
Denote by $\Sigma$ the closed strip $\{ s\in \cc: \Re s\in [0,1] \}
$. Suppose that $\{T_s\}_{s\in\Sigma}$ is a family of uniformly
bounded operators on $L^2$ such that the map $s\to \int_S
T_s(f)g\, d\rho$ is continuous on $\Sigma$ and analytic in the
interior of $\Sigma$, whenever $f,g\in L^2$. Moreover, assume that
there exists a positive constant $A$ such that
$$\|T_{it} f\|_{L^2}\le A \,\|f\|_{L^2}\qquad {\forall\ } f\in L^2,\,
{\forall\ } t\in\rr\,,$$
and
$$\|T_{1+it} f\|_{\BMO}\le A \,\|f\|_{\infty}\qquad {\forall\ } f\in
L^2\cap L^{\infty},\,{\forall\ } t\in\rr\,.$$
Then for every $\theta\in (0,1)$, the operator $T_{\theta}$ is
bounded on $L^{p_{\theta}}$, with
$\frac{1}{p_{\theta}}=\frac{1-\theta}{2}$ and
$$\|T_{\theta} f\|_{L^{p_{\theta}}}\le A_{\theta}
\,\|f\|_{L^{p_{\theta}}}\qquad {\forall \ } f\in L^2\cap
L^{p_{\theta}}\,.$$
Here $A_{\theta}$ depends only on $A$ and $\theta$.
\end{theorem}

\begin{proof}
This follows from Theorem \ref{int1}(i) and \cite[Theorem 1]{cj}.
Alternatively, we may follow the proof of \cite[p.\,175, Theorem 4]{st93}.
We leave the details to the reader.
\end{proof}

\begin{theorem}\label{t5.3}
Denote by $\Sigma$ the closed strip $\{ s\in \cc:\ \Re s\in [0,1] \}
$. Suppose that $\{T_s\}_{s\in\Sigma}$ is a family of uniformly
bounded operators on $L^2$ such that the map $s\to \int_S
T_s(f)g\, d\rho$ is continuous on $\Sigma$ and analytic in the
interior of $\Sigma$, whenever $f,g\in L^2$. Moreover, assume that
there exists a positive constant $A$ such that
$$\|T_{it} f\|_{L^1}\le A \,\|f\|_{H^1}\qquad {\forall\ } f\in
L^2\cap H^1,\,{\forall\ } t\in\rr\,,$$
and
$$\|T_{1+it} f\|_{{L^2}}\le A \,\|f\|_{{L^2}}\qquad {\forall\  } f\in
L^2,\,{\forall\ } t\in\rr\,.$$
Then for every $\theta\in (0,1)$, the operator $T_{\theta}$ is
bounded on $L^{q_{\theta}}$, with
$\frac{1}{q_{\theta}}=1-\frac{\theta}{2}$ and
$$
\|T_{\theta} f\|_{L^{q_{\theta}}}\le A_{\theta}
\,\|f\|_{L^{q_{\theta}}}\qquad {\forall\ } f\in L^2\cap
L^{q_{\theta}}\,.
$$
Here $A_{\theta}$ depends only on $A$ and $\theta$.
\end{theorem}

\begin{proof}
This follows from Theorem \ref{int1}(ii) and \cite[Theorem 1]{cj}.
We omit the details.
\end{proof}

\begin{remark}\label{dyadic-complex-interpolation}\rm
It is easy to see that Theorems 5.1, 5.2 and 5.3 still hold if $H^1$
and $\BMO$ are replaced by $H_\cp^1$ and $\BMO_\cp$, respectively.
We leave the details to the reader.
\end{remark}

\medskip

{\small\noindent{\bf Acknowledgments}\quad
Maria Vallarino is partially supported by PRIN 2007 ``Analisi Armonica"
and Dachun Yang (the corresponding author) is supported by the National
Natural Science Foundation (Grant No. 10871025) of China.

Also, the authors sincerely wish to express their
deeply thanks to the referee for her/his very carefully reading and
also her/his so many careful, valuable and suggestive remarks which
essentially improve the presentation of this article.}

\end{document}